\documentclass[]{article}

\usepackage{epsfig}
\usepackage{amssymb}

\newcommand\free{\texttt{Free}}

\newcommand\rel{\texttt{Rel}}
\newcommand\parts{\texttt{Parts}}
\newcommand\var{\texttt{Var}}
\newcommand\reach{\texttt{reach}}

\newcommand\domain{\texttt{Dom}}
\newcommand\dom\domain

\newcommand{\tuple}[1]{\vec{#1}}

\newtheorem{theo}{Theorem}[section]  
\newtheorem{coro}[theo]{Corollary}
\newtheorem{propo}{Proposition}[section]

\newtheorem{defin}[theo]{Definition}

\newenvironment{proof}{\emph{Proof:}$\\$}{$\\\Box\\$}

\begin{document}
\title{Transition Semantics\\\large{The Dynamics of Dependence Logic}}
\author{Pietro Galliani\\University of Helsinki\\(pgallian@gmail.com)}
\maketitle
\begin{abstract}
We examine the relationship between Dependence Logic and game logics. A variant of Dynamic Game Logic, called \emph{Transition Logic}, is developed, and we show that its relationship with Dependence Logic is comparable to the one between First-Order Logic and Dynamic Game Logic discussed by van Benthem. 

This suggests a new perspective on the interpretation of Dependence Logic formulas, in terms of assertions about \emph{reachability} in games of imperfect information against Nature. We then capitalize on this intuition by developing expressively equivalent variants of Dependence Logic in which this interpretation is taken to the foreground. 
\end{abstract}
\section{Introduction}
\subsection{Dependence Logic}
\label{subsect:DL}
Dependence Logic \cite{vaananen07} is an extension of First-Order Logic which adds \emph{dependence atoms} of the form $=\!\!\!(t_1, \ldots, t_n)$ to it, with the intended interpretation of ``the value of the term $t_n$ is a function of the values of the terms $t_1 \ldots t_{n-1}$.''

The introduction of such atoms is roughly equivalent to the introduction of non-linear patterns of dependence and independence between variables of Branching Quantifier Logic \cite{henkin61} or Independence Friendly Logic \cite{hintikkasandu89,hintikka96,mann11}: for example, both the Branching Quantifier Logic sentence
\[
	\left(
		\begin{array}{l l}
			\forall x & \exists y\\
			\forall z & \exists w
		\end{array}
	\right)
	R(x, y, z, w)
\]
and the Independence Friendly Logic sentence
\[
	\forall x \exists y \forall z (\exists w / x, y) R(x, y, z, w)
\]
correspond in Dependence Logic to
\[
	\forall x \exists y \forall z \exists w (=\!\!(z, w) \wedge R(x,y,z,w)),
\]
in the sense that all of these expressions are equivalent to the Skolem formula
\[
	\exists f \exists g \forall x \forall z R(x, f(x), z, g(z)).
\]
As this example illustrates, the main peculiarity of Dependence Logic compared to the others above-mentioned logics lies in the fact that, in Dependence Logic, the notion of \emph{dependence and independence between variables} is explicitly separated from the notion of quantification. This makes it an eminently suitable formalism for the formal analysis of the properties of \emph{dependence itself} in a first-order setting, and some recent papers (\cite{gradel13,engstrom12,galliani12}) explore the effects of replace dependence atoms with other similar primitives such as \emph{independence atoms} \cite{gradel13}, \emph{multivalued dependence atoms} \cite{engstrom12}, or \emph{inclusion} or \emph{exclusion} atoms \cite{galliani11b,galliani12}.

Branching Quantifier Logic, Independence Friendly Logic and Dependence Logic, as well as their variants, are called \emph{logics of imperfect information}: indeed, the truth conditions of their sentences can be obtained by defining, for every model $M$ and sentence $\phi$, an imperfect-information \emph{semantic game} $G^M(\phi)$ between a \emph{Verifier} (also called Eloise) and a \emph{Falsifier} (also called Abelard), and then asserting that $\phi$ is true in $M$ if and only if the Verifier has a winning strategy in $G^M(\phi)$. As an alternative of this (non-compositional) \emph{Game-Theoretic Semantics}, which is an imperfect-information variant of Hintikka's Game-Theoretic Semantics for First Order Logic \cite{hintikka68}, Hodges introduced in \cite{hodges97} \emph{Team Semantics} (also called \emph{Trump Semantics}), a compositional semantics for logics of imperfect information which is equivalent to Game-Theoretic Semantics over sentences and in which formulas are satisfied or not satisfied not by single assignments, but by \emph{sets} of assignments (called \emph{Teams}). 

In this work, we will be mostly concerned with Team Semantics and some of its variants. We refer the reader to the relevant literature (for example to \cite{vaananen07} and \cite{mann11}) for further information regarding these logics: in the rest of this section, we will content ourselves with recalling the definitions and results which will be useful for the rest of this work. 
\begin{defin}[Assignments and substitutions]
Let $M$ be a first order model and let $V$ be a finite set of variables. Then an \emph{assignment} over $M$ with \emph{domain} $V$ is a function $s$ from $V$ to the set $\dom(M)$ of all elements of $M$.

Furthermore, for any assignment $s$ over $M$ with domain $V$, any element $m \in \dom(M)$ and any variable $v$ (not necessarily in $V$), we write $s[m/v]$ for the assignment with domain $V \cup \{v\}$ such that 
\[
	s[m/v](w) = \left\{\begin{array}{l l}
			m & \mbox{if } w = v;\\
			s(w) & \mbox{if } w \in V \backslash \{v\}
		\end{array}
	\right. 
\]
for all $w \in V \cup \{v\}$.
\end{defin}
\begin{defin}[Team]
Let $M$ be a first-order model and let $V$ be a finite set of variables. A \emph{team} $X$ over $M$ with \emph{domain} $\dom(X) = V$ is a set of assignments from $V$ to $M$. 
\end{defin}
\begin{defin}[Relations corresponding to teams]
Let $X$ be a team over $M$, and let $V$ be a finite set of variables.  and let $\tuple v$ be a finite tuple of variables in its domain. Then $X(\tuple v)$ is the relation $\{s(\tuple v) : s \in X\}$. Furthermore, we write $\rel(X)$ for $X(\domain(X))$.  
\end{defin}
As is often the case for Dependence Logic, we will assume that all our formulas are in Negation Normal Form:
\begin{defin}[Dependence Logic, Syntax]
Let $\Sigma$ be a first-order signature. Then the set of all dependence logic formula with signature $\Sigma$ is given by
\[
\phi ::= R\tuple t ~|~ \lnot R \tuple t ~|~ =\!\!(t_1, \ldots, t_n) ~|~ \phi \vee \phi ~|~ \phi \wedge \phi ~|~ \exists v \phi ~|~ \forall v \phi	
\]
where $R$ ranges over all relation symbols, $\tuple t$ ranges over all tuples of terms of the appropriate arities, $t_1 \ldots t_n$ range over all terms and $v$ ranges over the set $\var$ of all variables.  
\end{defin}
The set $\free(\phi)$ of all \emph{free variables} of a formula $\phi$ is defined precisely as in First Order Logic, with the additional condition that all variables occurring in a dependence atom are free with respect to it. 
\begin{defin}[Dependence Logic, Semantics]
\label{DL-TS}
Let $M$ be a first-order model, let $X$ be a team over it, and let $\phi$ be a Dependence Logic formula with the same signature of $M$ and with free variables in $\domain(X)$. Then we say that $X$ \emph{satisfies} $\phi$ in $M$, and we write $M \models_X \phi$, if and only if 
\begin{description}
\item[TS-lit: ] $\phi$ is a first-order literal and $M \models_s \phi$ for all $s \in X$;
\item[TS-dep: ] $\phi$ is a dependence atom $=\!\!(t_1, \ldots, t_n)$ and any two assignments $s, s' \in X$ which assign the same values to $t_1 \ldots t_{n-1}$ also assign the same value to $t_n$; 
\item[TS-$\vee$: ] $\phi$ is of the form $\psi_1 \vee \psi_2$ and there exist two teams $Y_1$ and $Y_2$ such that $X = Y_1 \cup Y_2$, $M \models_{Y_1} \psi_1$ and $M \models_{Y_2} \psi_2$;
\item[TS-$\wedge$: ] $\phi$ is of the form $\psi_1 \wedge \psi_2$, $M \models_X \psi_1$ and $M \models_X \psi_2$; 
\item[TS-$\exists$: ] $\phi$ is of the form $\exists v \psi$ and there exists a function $F: X \rightarrow \domain(M)$ such that $M \models_{X[F/v]} \psi$, where 
\[
	X[F/v] = \{s[F(s)/v] : s \in X\}
\]
\item[TS-$\forall$: ] $\phi$ is of the form $\forall v \psi$ and $M \models_{X[M/v]} \psi$, where 
\[
	X[M/v] = \{s[m/v] : s \in X, m \in \domain(M)\}.
\]
\end{description}
\end{defin}
The disjunction of Dependence Logic does not behave like the classical disjunction: for example, it is easy to see that $=\!\!(x) \vee =\!\!(x)$ is not equivalent to $=\!\!(x)$, as the former holds for the team $X = \{\{(x, 0)\}, \{(x, 1)\}\}$ and the latter does not. However, it is possible to define the classical disjunction in terms of the other connectives:
\begin{defin}[Classical Disjunction]
\label{defin:classic_or}
Let $\psi_1$ and $\psi_2$ be two Dependence Logic formulas, and let $u_1$ and $u_2$ be two variables not occurring in them. Then we write $\psi_1 \sqcup \psi_2$ as a shorthand for 
\[
	\exists u_1 \exists u_2 (=\!\!(u_1) \wedge =\!\!(u_2) \wedge ((u_1 = u_2 \wedge \psi_1) \vee (u_1 \not = u_2 \wedge \psi_2))).
\]
\end{defin}
\begin{propo}
\label{propo:classic_or}
For all formulas $\psi_1$ and $\psi_2$, all models $M$ with at least two elements\footnote{In general, we will assume through this whole work that all first-order models which we are considering have at least two elements. As one-element models are trivial, this is not a very onerous restriction.} whose signature contains that of $\psi_1$ and $\psi_2$ and all teams $X$ whose domain contains the free variables of $\psi_1$ and $\psi_2$ 
\[
	M \models_X \psi_1 \sqcup \psi_2 \Leftrightarrow M \models_X \psi_1 \mbox{ or } M \models_X \psi_2.
\]
\end{propo}
The following four proportions are from \cite{vaananen07}:
\begin{propo}
\label{propo:emptyteam}
For all models $M$ and Dependence Logic formulas $\phi$, $M \models_\emptyset \phi$.
\end{propo}
\begin{propo}[Downwards Closure]
If $M \models_X \phi$ and $Y \subseteq X$ then $M \models_Y \psi$.
\end{propo}
\begin{propo}[Locality]
If $M \models_X \phi$ and $X(\free(\phi)) = Y(\free(\phi))$ then $M \models_Y \phi$.
\end{propo}
\begin{propo}[From Dependence Logic to $\Sigma_1^1$]
\label{DLToSigma}
Let $\phi(\tuple v)$ be a Dependence Logic formula with free variables in $\tuple v$. Then there exists a $\Sigma_1^1$ sentence $\Phi(R)$ such that 
\[
	M \models_X \phi \Leftrightarrow M \models \Phi(X(\tuple v))
\]
for all suitable models $M$ and for all nonempty teams $X$. Furthermore, in $\Phi(R)$ the symbol $R$ occurs only negatively. 
\end{propo}
As proved in \cite{kontinenv09}, there is also a converse for the last proposition:
\begin{theo}[From $\Sigma_1^1$ to Dependence Logic]
\label{SigmaToDL}
Let $\Phi(R)$ be a $\Sigma_1^1$ sentence in which $R$ occurs only negatively. Then there exists a Dependence Logic formula $\phi(\tuple v)$, where $|\tuple v|$ is the arity of $R$, such that 
\[
	M \models_X \phi \Leftrightarrow M \models \Phi(X(\tuple v))
\]
for all suitable models $M$ and for all nonempty teams $X$ whose domain contains $\tuple v$.
\end{theo}
Because of this correspondence between Dependence Logic and Existential Second Order Logic, it is easy to see that Dependence Logic is closed under existential quantification: for all Dependence Logic formulas $\phi(\tuple v, P)$ over the signature $\Sigma \cup \{P\}$ there exists a Dependence Logic formula $\exists P \phi(\tuple v, P)$ over the signature $\Sigma$ such that
\[
	M \models_X \exists P \phi(\tuple v, P) \Leftrightarrow \exists P \mbox{ s.t. } M \models_X \phi(\tuple v, P)
\]
for all models $M$ with domain $\Sigma$ and for all teams $X$ over the free variables of $\phi$. Therefore, in the rest of this work we will add second-order existential quantifiers to the language of Dependence Logic, and we will write $\exists P \phi(\tuple v, P)$ as a shorthand for the corresponding Dependence Logic expression.
\subsection{Dynamic Game Logic} 
\label{subsect:DGL}
\emph{Game logics} are logical formalisms for reasoning about games and their properties in a very general setting. Whereas the Game Theoretic Semantics approach attempts to use game-theoretic techniques to \emph{interpret} logical systems, game logics attempt to put logic to the service of game theory, by providing a high-level language for the study of games.

They generally contain two different kinds of expressions: 
\begin{enumerate}
\item \emph{Game terms}, which are descriptions of games in terms of compositions of certain primitive \emph{atomic games}, whose interpretation is presumed fixed for any given game model;
\item \emph{Formulas}, which, in general, correspond to assertions about the abilities of players in games. 
\end{enumerate}

In this subsection, we are going to summarize the definition of a variant of Dynamic Game Logic \cite{parikh85}.\footnote{The main difference between this version and the one of Parikh's original paper lies in the absence of the iteration operator $\gamma^*$ from our formalism. In this, we follow \cite{vanbenthem03,vanbenthem08b}.} Then, in the next subsection, we will discuss a remarkable connection between First-Order Logic and Dynamic Game Logic discovered by Johan van Benthem in \cite{vanbenthem03}.\\

One of the fundamental semantic concepts of Dynamic Game Logic is the notion of \emph{forcing relation:}
\begin{defin}[Forcing Relation]
Let $S$ be a nonempty set of \emph{states}. A \emph{forcing relation} over $S$ is a set $\rho \subseteq S \times \parts(S)$, where $\parts(S)$ is the powerset of $S$.
\end{defin}
In brief, a forcing relation specifies the abilities of a player in a perfect-information game: $(s, X) \in \rho$ if and only if the player has a strategy that guarantees that, whenever the initial position of the game is $s$, the terminal position of the game will be in $X$. 

A (two-player) \emph{game} is then defined as a pair of forcing relations satisfying some axioms:
\begin{defin}[Game]
Let $S$ be a nonempty set of states. A \emph{game} over $S$ is a pair $(\rho^E, \rho^A)$ of forcing relations over $S$ satisfying the following conditions for all $i \in \{E, A\}$, all $s \in S$ and all $X, Y \subseteq S$:
\begin{description}
\item[Monotonicity: ] If $(s, X) \in \rho^i$ and $X \subseteq Y$ then $(s, Y) \in \rho^i$; 
\item[Consistency: ] If $(s, X) \in \rho^E$ and $(s, Y) \in \rho^A$ then $X \cap Y \not = \emptyset$; 
\item[Non-triviality: ] $(s, \emptyset) \not \in \rho^i$. 
\item[Determinacy: ] If $(s, X) \not \in \rho^i$ then $(s, S \backslash X) \in \rho^j$, where $j \in \{E, A\} \backslash \{i\}$.\footnote{This requirement is nothing but a formal version of Zermelo's Theorem: if one of the players cannot force the outcome of the game to belong to a set of ``winning outcomes'' $X$, this implies that the other player can force it to belong to the complement of $X$.}
\end{description}
\end{defin}
\begin{defin}[Game Model]
Let $S$ be a nonempty set of states, let $\Phi$ be a nonempty set of \emph{atomic propositions} and let $\Gamma$ be a nonempty set of \emph{atomic game symbols}. Then a \emph{game model} over $S$, $\Phi$ and $\Gamma$ is a triple $(S, \{(\rho^E_g, \rho^A_g) : g \in \Gamma\}, V)$, where $(\rho^E_g, \rho^A_g)$ is a game over $S$ for all $g \in \Gamma$ and where $V$ is a valutation function associating each $p \in \Phi$ to a subset $V(p) \subseteq S$. 
\end{defin}
The language of Dynamic Game Logic, as we already mentioned, consists of \emph{game terms}, built up from atomic games, and of \emph{formulas}, built up from atomic proposition. The connection between these two parts of the language is given by the \emph{test} operation $\phi?$, which turns any formula $\phi$ into a test game, and the \emph{diamond} operation, which combines a game term $\gamma$ and a formula $\phi$ into a new formula $\langle \gamma, i \rangle \phi$ which asserts that agent $i$ can guarantee that the game $\gamma$ will end in a state satisfying $\phi$.
\begin{defin}[Dynamic Game Logic - Syntax]
Let $\Phi$ be a nonempty set of \emph{atomic propositions} and let $\Gamma$ be a nonempty set of \emph{atomic game formulas}. Then the sets of all game terms $\gamma$ and formulas $\phi$ are defined as 
\begin{eqnarray*}
\gamma & ::= & g ~|~ \phi? ~|~ \gamma;\gamma ~|~ \gamma \cup \gamma ~|~ \gamma^d\\
\phi & ::= & \bot ~|~ p ~|~ \lnot \phi ~|~ \phi \vee \phi ~|~ \langle \gamma, i \rangle \phi
\end{eqnarray*}
for $p$ ranging over $\Phi$, $g$ ranging over $\Gamma$, and $i$ ranging over $\{E, A\}$.
\end{defin}
We already mentioned the intended interpretations of the test connective $\phi?$ and of the diamond connective $\langle \gamma, i \rangle \phi$. The interpretations of the other game connectives should be clear: $\gamma^d$ is obtained by swapping the roles of the players in $\gamma$, $\gamma_1 \cup \gamma_2$ is a game in which the existential player $E$ chooses whether to play $\gamma_1$ or $\gamma_2$, and $\gamma_1;\gamma_2$ is the \emph{concatenation} of the two games corresponding to $\gamma_1$ and $\gamma_2$ respectively. 
\begin{defin}[Dynamic Game Logic - Semantics]
Let $G = (S, \{(\rho^E_g, \rho^A_g) : g \in \Gamma\}, V)$ be a game model over $S$, $\Gamma$ and $\Phi$. Then for all game terms $\gamma$ and all formulas $\phi$ of Dynamic Game Logic over $\Gamma$ and $\Phi$ we define a game $\|\gamma\|_G$ and a set $\|\phi\|_G \subseteq S$ as follows: 
\begin{description}
\item[DGL-atomic-game: ] For all $g \in \Gamma$, $\|g\|_G = (\rho^E_g, \rho^A_g)$;
\item[DGL-test: ]For all formulas $\phi$, $\|\phi?\|_G = (\rho^E, \rho^A)$, where 
\begin{itemize}
\item $s \rho^E X$ iff $s \in \|\phi\|_G$ and $s \in X$;\\
\item $s \rho^A X$ iff $s \not \in \|\phi\|_G$ or $s \in X$
\end{itemize}
for all $s \in S$ and all $X$ with $\emptyset \not = X \subseteq S$; 
\item[DGL-concat: ] For all game terms $\gamma_1$ and $\gamma_2$, $\|\gamma_1;\gamma_2\|_G = (\rho^E, \rho^A)$, where, for all $i \in \{E, A\}$ and for $\|\gamma_1\|_G = (\rho^E_1, \rho^A_1)$, $\|\gamma_2\|_G = (\rho^E_2, \rho^A_2)$, 
\begin{itemize}
\item $s \rho^i X$ if and only if there exists a $Z$ such that $s \rho^i_1 Z$ and for each $z \in Z$ there exists a set $X_z$ satisfying $z \rho^i_2 X_z$ such that
\[
	X = \bigcup_{z \in Z} X_z; 
\]
\end{itemize}
\item[DGL-$\cup$: ] For all game terms $\gamma_1$ and $\gamma_2$, $\|\gamma_1 \cup \gamma_2\|_G = (\rho^E, \rho^A)$, where 
\begin{itemize}
\item $s \rho^E X$ if and only if $s \rho^E_1 X$ or $s \rho^E_2 X$, and 
\item $s \rho^A X$ if and only if $s \rho^A_1 X$ and $s \rho^A_2 X$
\end{itemize}
where, as before, $\|\gamma_1\|_G = (\rho^E_1, \rho^A_1)$ and $\|\gamma_2\|_G = (\rho^E_2, \rho^A_2)$;\footnote{\cite{vanbenthem08b} gives the following alternative condition for the powers of the universal player: 
\begin{itemize}
\item $s \rho^A X$ if and only if $X = Z_1 \cup Z_2$ for two $Z_1$ and $Z_2$ such that $s \rho^A_1 Z_1$ and $s \rho^A_2 Z_2$. 
\end{itemize}
It is trivial to see that, if our games satisfy the monotonicity condition, this rules is equivalent to the one we presented.}
\item[DGL-dual: ] If $\|\gamma\|_G = (\rho^E, \rho^A)$ then $\|\gamma^d\|_G = (\rho^A, \rho^E)$; 
\item[DGL-$\bot$: ] $\|\bot\|_G = \emptyset$; 
\item[DGL-atomic-pr: ] $\|p\|_G = V(p)$; 
\item[DGL-$\lnot$: ] $\|\lnot \phi \|_G = S \backslash \|\phi\|_G$; 
\item[DGL-$\vee$: ] $\|\phi_1 \vee \phi_2\|_G = \|\phi_1\|_G \cup \|\phi_2\|_G$;
\item[DGL-$\diamond$: ] If $\|\gamma\|_G = (\rho^E, \rho^A)$ then for all $\phi$, 
\[
	\|\langle \gamma, i \rangle \phi\|_G = \{s \in S : \exists X_s \subseteq \|\phi\|_G \mbox{ s.t. } s \rho^i X_s\}.
\]
\end{description}
If $s \in \|\phi\|_G$, we say that $\phi$ is \emph{satisfied} by $s$ in $G$ and we write $M \models_s \phi$. 
\end{defin}
We will not discuss here the properties of this logic, or the vast amount of variants and extensions of it which have been developed and studied. It is worth pointing out, however, that \cite{vanbenthem08b} introduced a \emph{Concurrent Dynamic Game Logic} that can be considered one of the main sources of inspiration for the Transition Logic that we will develop in Subsection \ref{subsect:TDL}.
\subsection{Dynamic Game Logic and First Order Logic}
\label{subsect:DGL-Rep}
In this subsection, we will briefly recall a remarkable result from \cite{vanbenthem03} which establishes a connection between Dynamic Game Logic and First-Order Logic. 

In brief, as the following two theorems demonstrate, either of these logics can be seen as a special case of the other, in the sense that models and formulas of the one can be uniformly translated into models of the other in a way which preserves satisfiability and truth:
\begin{theo}
\label{theo:repFO1}
Let $G = (S, \{(\rho^E_g, \rho^A_g) : g \in \Gamma\}, V)$ be any game model, let $\phi$ be any game formula for the same language, and let $s \in S$. Then it is possible to uniformly construct a first-order model $G^{FO}$, a first-order formula $\phi^{FO}$ and an assignment $s^{FO}$ of $G^{FO}$ such that 
\[
	G \models_s \phi \Leftrightarrow G^{FO} \models_{s^{FO}} \phi^{FO}.
\]
\end{theo}
\begin{theo}
\label{theo:repFO2}
Let $M$ be any first order model, let $\phi$ be any first-order formula for the signature of $M$, and let $s$ be an assignment of $M$. Then it is possible to uniformly construct a game model $G^{DGL}$, a game formula $\phi^{DGL}$ and a state $s^{DGL}$ such that 
\[
	M \models_s \phi \Leftrightarrow G^{DGL} \models_{s^{DGL}} \phi^{DGL}.
\]
\end{theo}
We will not discuss here the proofs of these two results. Their \emph{significance}, however, is something about which is necessary to spend a few words. In brief, what this back-and-forth representation between First Order Logic and Dynamic Game Logic tells us is that it is possible to understand First Order Logic as a \emph{logic for reasoning about determined games}!

In the next sections, we will attempt to develop a similar result for the case of Dependence Logic.
\section{Transition Logic}
\subsection{A Logic for Imperfect Information Games Against Nature}
\label{subsect:TL}
We will now define a variant of Dynamic Game Logic, which we will call \emph{Transition Logic}. It deviates from the basic framework of Dynamic Game Logic in two fundamental ways: 
\begin{enumerate}
\item It considers \emph{one-player} games against Nature, instead of \emph{two-player games} as is usual in Dynamic Game Logic; 
\item It allows for \emph{uncertainty} about the initial position of the game. 
\end{enumerate}
Hence, Transition Logic can be seen as a \emph{decision-theoretic logic}, rather than a \emph{game-theoretic} one: Transition Logic formulas, as we will see, correspond to assertions about the abilities of a single agent acting under uncertainty, instead of assertions about the abilities of agents interacting with each other.

In principle, it is certainly possible to generalize the approach discussed here to multiple agents acting in situations of imperfect information, and doing so might cause interesting phenomena to surface; but for the time being, we will content ourselves with developing this formalism and discussing its connection with Dependence Logic.

Our first definition is a fairly straightforward generalization of the concept of forcing relation:
\begin{defin}[Transition system]
Let $S$ be a nonempty set of \emph{states}. A \emph{transition system} over $S$ is a nonempty relation $\theta \subseteq \parts(S) \times \parts(S)$ satisfying the following requirements: 
\begin{description}
\item[Downwards Closure: ] If $(X, Y) \in \theta$ and $X' \subseteq X$ then $(X', Y) \in \theta$;
\item[Monotonicity: ] If $(X, Y) \in \theta$ and $Y \subseteq Y'$ then $(X, Y') \in \theta$;
\item[Non-creation: ] $(\emptyset, Y) \in \theta$ for all $Y \subseteq S$; 
\item[Non-triviality: ] If $X \not = \emptyset$ then $(X, \emptyset) \not \in \theta$.
\end{description}
\end{defin}
Informally speaking, a transition system specifies the abilities of an agent: for all $X, Y \subseteq S$ such that $(X, Y) \in \theta$, the agent has a strategy which guarantees that the output of the transition will be in $Y$ whenever the input of the transition is in $X$. 

The four axioms which we gave capture precisely this intended meaning, as we will see:
\begin{defin}[Decision Game]
A \emph{decision game} is a triple $\Gamma = (S, E, O)$, where $S$ is a nonempty set of \emph{states}, $E$ is a nonempty set of possible \emph{decisions} for our agent and $O$ is an \emph{outcome function} from $S \times E$ to $\parts(S)$. 

If $s' \in O(s, e)$, we say that $s'$ is a \emph{possible outcome} of $s$ under $e$; if $O(s, e) = \emptyset$, we say that $e$ \emph{fails} on input $s$. 
\end{defin}
\begin{defin}[Abilities in a decision game]
Let $\Gamma = (S, E, O)$ be a decision game, and let $X, Y \subseteq S$. Then we say that $\Gamma$ \emph{allows} the transition $X \rightarrow Y$, and we write $\Gamma: X \rightarrow Y$, if and only if there exists a $e \in E$ such that $\emptyset \not = O(s, e) \subseteq Y$ for all $s \in X$ (that is, if and only if our agent can make a decision which guarantees that the outcome will be in $Y$ whenever the input is in $X$).
\end{defin}
\begin{theo}[Transition Systems and Abilities]
A set $\theta \subseteq \parts(S) \times \parts(S)$ is a transition system if and only if there exists a decision game $\Gamma = (S, E, O)$ such that 
\[
	(X, Y) \in \theta \Leftrightarrow \Gamma:X \rightarrow Y.
\]
\end{theo}
\begin{proof}
Let $\theta \subseteq \parts(S) \times \parts(S)$ be any transition system, let us enumerate its elements $\{(X_i, Y_i) : i \in I)\}$, and let us consider the game $\Gamma = (S, I, O)$, where
\[
	O(s, i) = \left\{\begin{array}{l l}
		Y_i & \mbox{ if } s \in X_i;\\	
		\emptyset & \mbox{ otherwise.}
	\end{array}
	\right.
\]
Suppose that $(X, Y) \in \theta$. If $X = \emptyset$, then $\Gamma:X \rightarrow Y$ follows at once by definition. If instead $X \not = \emptyset$, by \textbf{non-triviality} we have that $Y$ is nonempty too, and furthermore $(X, Y) = (X_i, Y_i)$ for some $i \in I$. Then $O(s, i) = Y_i \not= \emptyset$ for all $s \in X_i$, as required. 

Now suppose that $\Gamma: X \rightarrow Y$. Then there exists a $i \in I$ such that $\emptyset \not = O(s, i) \subseteq Y$ for all $s \in X$. If $X \not = \emptyset$, this implies that $X \subseteq X_i$ and $Y_i \subseteq Y$. Hence, by \textbf{monotonicity} and \textbf{downwards closure}, $(X, Y) \in \theta$, as required. If instead $X = \emptyset$, then by \textbf{non-creation} we have again that $(X, Y) \in \theta$. 

Conversely, consider a decision game $\Gamma = (S, E, O)$. Then the set of its abilities satisfies our four axioms: 
\begin{description}
\item[Downwards Closure: ]Suppose that $\Gamma : X \rightarrow Y$ and that $X' \subseteq X$. By definition, there exists a $e \in E$ such that $\emptyset \not = O(s, e) \subseteq Y$ for all $s \in X$. But then the same holds for all $s \in X'$, and hence $\Gamma: X' \rightarrow Y$.
\item[Monotonicity: ]Suppose that $\Gamma : X \rightarrow Y$ and that $Y \subseteq Y'$. By definition, there exists a $e \in E$ such that $\emptyset \not = O(s, e) \subseteq Y$ for all $s \in X$. But then, for all such $s$, $O(s,e) \subseteq Y'$ too, and hence $\Gamma: X \rightarrow Y'$.
\item[Non-creation: ] Let $Y \subseteq S$ and let $e \in E$ be any possible decision. Then trivially $\emptyset \not = O(s, e) \subseteq Y$ for all $s \in \emptyset$, and hence $\Gamma : \emptyset \rightarrow Y$. 
\item[Non-triviality: ] Let $s_0 \in X$, and suppose that $\Gamma: X \rightarrow Y$. Then there exists a $e$ such that $\emptyset \not = O(s, e) \subseteq Y$ for all $s \in X$, and hence in particular $\emptyset \not = O(s_0, e) \subseteq Y$. Therefore, $Y$ is nonempty. 
\end{description}
\end{proof}
What this theorem tells us is that our notion of transition system is the correct one: it captures precisely the abilities of an agent making choices under imperfect information and attempting to guarantee that, if the initial state is in a set $X$, the outcome will be in a set $Y$. 
 
\begin{defin}[Trump]
Let $S$ be a nonempty set of states. A \emph{trump} over $S$ is a nonempty, downwards closed family of subsets of $S$. 
\end{defin}
Whereas a transition system describes the abilities of an agent to transition from a set of possible initial states to a set of possible terminal states, a trump describes the agent's abilities to reach \emph{some} terminal state from a set of possible initial states:\footnote{The term ``trump'' is taken from \cite{hodges97}, where it is used to describe the set of all teams which satisfy a given formula.}
\begin{propo}
Let $\theta$ be a transition system and let $Y \subseteq S \not = \emptyset$. Then $\reach(\theta, Y) = \{X ~|~ (X, Y) \in \theta\}$ forms a trump. Conversely, for any trump $\mathcal X$ over $S$ there exists a transition system $\theta$ such that $\mathcal X = \reach(\theta, Y)$ for any nonempty $Y \subseteq S$.
\end{propo}
\begin{proof}
Let $\theta$ be a transition system. Then if $(X, Y) \in \theta$ and $X' \subseteq X$, by downwards closure we have at once that $(X', Y) \in \theta$. Furthermore, $(\emptyset, Y) \in \theta$ for any $Y$. Hence, $\reach(\theta, Y)$ is a trump, as required.

Conversely, let $\mathcal X \subseteq \parts(\parts(S))$ be a trump, and let us enumerate its elements as $\{X_i : i \in I\}$. Then define $\theta$ as
\[
	\theta = \{(A, B) : \emptyset \not = B \subseteq S, \exists i \in I \mbox{ s.t. } A \subseteq X_i\} \cup \{(\emptyset, \emptyset)\}
\]
It is easy to see that $\theta$ is a transition system; and by construction, for $Y \not = \emptyset$ we have that $(A, Y) \in \theta \Leftrightarrow \exists i \mbox{ s.t. } A \subseteq X_i \Leftrightarrow A \in \mathcal X$, where we used the fact that $\mathcal X$ is downwards closed.
\end{proof}

We can now define the syntax and semantics of Transition Logic:
\begin{defin}[Transition Model]
Let $\Phi$ be a set of \emph{atomic propositional symbols} and let $\Theta$ be a set of \emph{atomic transition symbols}. Then a \emph{transition model} is a tuple $T = (S, \{\theta_t : t \in \Theta\}, V)$, where $S$ is a nonempty set of states, $\theta_t$ is a transition system over $S$ for any $t \in \Theta$, and $V$ is a function sending each $p \in \Phi$ into a trump of $S$.
\end{defin}
\begin{defin}[Transition Logic - Syntax]
Let $\Phi$ be a set of atomic propositions and let $\Theta$ be a set of atomic transitions. Then the \emph{transition terms} and \emph{formulas} of our language are defined respectively as 
\begin{eqnarray*}
\tau &::=& t ~|~ \phi? ~|~ \tau \otimes \tau ~|~ \tau \cap \tau ~|~ \tau;\tau\\
\phi &::=& \top ~|~ p ~|~ \phi \vee \phi ~|~ \phi \wedge \phi ~|~ \langle \tau\rangle \phi
\end{eqnarray*}
where $t$ ranges over $\Theta$ and $p$ ranges over $\Phi$.
\end{defin}
\begin{defin}[Transition Logic - Semantics]
Let $T = (S, \{\theta_t : t \in \Theta), V)$ be a transition model, let $\tau$ be a transition term,  and let $X, Y \subseteq S$. Then we say that $\tau$ \emph{allows} the transition from $X$ to $Y$, and we write $T \models_{X \rightarrow Y} \tau$, if and only if 
\begin{description}
\item[TL-atomic-tr: ] $\tau = t$ for some $t \in \Theta$ and $(X, Y) \in \theta_t$; 
\item[TL-test: ] $\tau = \phi?$ for some transition formula $\phi$ such that $T \models_X \phi$ in the sense described later in this definition, and $X \subseteq Y$; 
\item[TL-$\otimes$: ] $\tau = \tau_1 \otimes \tau_2$, and $X = X_1 \cup X_2$ for two $X_1$ and $X_2$ such that $T \models_{X_1 \rightarrow Y} \tau_1$ and $T \models_{X_2 \rightarrow Y} \tau_2$; 
\item[TL-$\cap$: ] $\tau = \tau_1 \cap \tau_2$, $T \models_{X \rightarrow Y} \tau_1$ and $T \models_{X \rightarrow Y} \tau_2$; 
\item[TL-concat: ] $\tau = \tau_1;\tau_2$ and there exists a $Z \subseteq S$ such that $T \models_{X \rightarrow Z} \tau_1$ and $T \models_{Z \rightarrow Y} \tau_2$. 
\end{description}
Analogously, let $\phi$ be a transition formula, and let $X \subseteq S$. Then we say that $X$ \emph{satisfies} $\phi$, and we write $T \models_X \phi$, if and only if 
\begin{description}
\item[TL-$\top$: ]$\phi = \top$;
\item[TL-atomic-pr: ]$\phi = p$ for some $p \in \Phi$ and $X \in V(p)$; 
\item[TL-$\vee$: ]$\phi = \psi_1 \vee \psi_2$ and $T \models_X \psi_1$ or $T \models_X \psi_2$; 
\item[TL-$\wedge$: ] $\phi = \psi_1 \wedge \psi_2$, $T \models_X \psi_1$ and $T \models_X \psi_2$; 
\item[TL-$\diamond$: ]$\phi = \langle \tau \rangle \psi$ and there exists a $Y$ such that $T \models_{X \rightarrow Y} \tau$ and $T \models_Y \psi$.
\end{description}
\end{defin}
\begin{propo}
For any transition model $T$, transition term $\tau$ and transition formula $\phi$, the set 
\[
	\|\tau\|_T = \{(X, Y) : T \models_{X \rightarrow Y} \tau\}
\]
is a transition system and the set
\[
	\|\phi\|_T = \{X : T \models_X \phi\}
\]
is a trump.
\end{propo}
\begin{proof}
By induction.
\end{proof}
We end this subsection with a few simple observations about this logic.

First of all, we did not take the negation as one of the primitive connectives. Indeed, Transition Logic, much like Dependence Logic, has an intrinsically \emph{existential} character: it can be used to reason about which sets of possible states an agent \emph{may} reach, but not to reason about which ones such an agent \emph{must} reach. There is of course no reason, in principle, why a negation could not be added to the language, just as there is no reason why a negation  cannot be added to Dependence Logic, thus obtaining the far more powerful \emph{Team Logic} \cite{vaananen07b,kontinennu09}: however, this possible extension will not be studied in this work.

The connectives of Transition Logic are, for the most part, very similar to those of Dynamic Game Logic, and their interpretation should pose no difficulties. The exception is the \emph{tensor operator} $\tau_1 \otimes \tau_2$, which substitutes the game union operator $\gamma_1 \cup \gamma_2$ and which, while sharing roughly the same informal meaning, behaves in a very different way from the semantic point of view (for example, it is not in general idempotent!) 

The decision game corresponding to $\tau_1 \otimes \tau_2$ can be described as follows: first the agent chooses an index $i \in \{1, 2\}$, then he or she picks a strategy for $\tau_i$ and plays accordingly. However, the choice of $i$ may be a function of the initial state: hence, the agent can guarantee that the output state will be in $Y$ whenever the input state is in $X$ only if he or she can split $X$ into two subsets $X_1$ and $X_2$ and guarantee that the state in $Y$ will be reached from any state in $X_1$ when $\tau_1$ is played, and from any state in $X_2$ when $\tau_2$ is played.

It is also of course possible to introduce a ``true'' choice operator $\tau_1 \cup \tau_2$, with semantical condition
\begin{description}
\item[TL-$\cup$: ] $T \models_{X \rightarrow Y} \tau_1 \cup \tau_2$ iff $T \models_{X \rightarrow Y} \tau_1$ or $T \models_{X \rightarrow Y} \tau_2$;
\end{description}
but we will not explore this possibility any further in this work, nor we will consider any other possible connectives such as, for example, the iteration operator
\begin{description}
\item[TL-$*$: ] $T \models_{X \rightarrow Y} \tau^*$ iff there exist $n \in \mathbb N$ and $Z_0 \ldots Z_n$ such that $Z_0 = X$, $Z_n = Y$ and $T \models_{Z_i \rightarrow Z_{i+1}} \tau$ for all $i \in 1 \ldots n-1$.
\end{description}
\subsection{Transition Logic and Dependence Logic}
\label{subsect:TL-Rep}
This subsection contains the central result of this work, that is, the analogues of Theorems \ref{theo:repFO1} and \ref{theo:repFO2} for Dependence Logic and Transition Logic. \\

Representing Dependence Logic models and formulas in Transition Logic is fairly simple:
\begin{defin}[$M^{TL}$]
\label{defin:DL2TL-mod}
Let $M$ be a first-order model. Then $M^{TL}$ is the transition model $(S, \Theta, V)$ such that 
\begin{itemize}
\item $S$ is the set of all teams over $M$; 
\item The set of all atomic transition symbols is $\{\exists v, \forall v : v \in \var\}$, and hence $\Theta$ is $\{\theta_{\exists v}, \theta_{\forall v} : v \in \var\}$;
\item For any variable $v$, $\theta_{\exists v} = \{(X, Y) : \exists F \mbox{ s.t. } X[F/v] \subseteq Y \}$ and $\theta_{\forall v} = \{(X, Y) : X[M/v] \subseteq Y\}$;
\item For any first-order literal or dependence atom $\alpha$, $V(\alpha) = \{X : M \models_X \phi\}$. 
\end{itemize}
\end{defin}
\begin{defin}[$\phi^{TL}$]
\label{defin:DL2TL-form}
Let $\phi$ be a Dependence Logic formula. Then $\phi^{TL}$ is the transition term defined as follows: 
\begin{enumerate}
\item If $\phi$ is a literal or a dependence atom, $\phi^{TL} = \phi?$;
\item If $\phi = \psi_1 \vee \psi_2$, $\phi^{TL} = (\psi_1)^{TL} \otimes (\psi_2)^{TL}$; 
\item If $\phi = \psi_1 \wedge \psi_2$, $\phi^{TL} = (\psi_1)^{TL} \wedge (\psi_2)^{TL}$;
\item If $\phi = \exists v \psi$, $\phi^{TL} = \exists v; (\psi)^{TL}$; 
\item If $\phi = \forall v \psi$, $\phi^{TL} = \forall v; (\psi)^{TL}$. 
\end{enumerate}
\end{defin}
\begin{theo}
\label{theo:TL-Rep1}
For all first-order models $M$, teams $X$ and formulas $\phi$, the following are equivalent: 
\begin{itemize}
	\item $M \models_X \phi$;
	\item $\exists Y \mbox{ s.t. } M^{TL} \models_{X \rightarrow Y} \phi^{TL}$;
	\item $M^{TL} \models_X \langle \phi^{TL}\rangle \top$;
	\item $M^{TL} \models_{X \rightarrow S} \phi^{TL}$.
\end{itemize}
\end{theo}
\begin{proof}
We show, by structural induction on $\phi$, that the first condition is equivalent to the last one. The equivalences between the last one and the second and third ones are then trivial.
\begin{enumerate}
\item If $\phi$ is a literal or a dependence atom, $M^{TL} \models_{X \rightarrow S} \phi?$ if and only if $X \in V(\phi)$, that is, if and only if $M \models_X \phi$;
\item $M^{TL} \models_{X \rightarrow S} (\psi_1)^{TL} \otimes (\psi_2)^{TL}$ if and only if $X = X_1 \cup X_2$ for two $X_1, X_2 \subseteq S$ such that $M^{TL} \models_{X_1 \rightarrow S} (\psi_1)^{TL}$ and $M^{TL} \models_{X_2 \rightarrow S} (\psi_2)^{TL}$. By induction hypothesis, this can be the case if and only if $M \models_{X_1} \psi_1$ and $M \models_{X_2} \psi_2$, that is, if and only if $M \models_X \psi_1 \vee \psi_2$. 
\item $M^{TL} \models_{X \rightarrow S} (\psi_1)^{TL} \wedge (\psi_2)^{TL}$ if and only if $M^{TL} \models_{X \rightarrow S} (\psi_1)^{TL}$ and $M^{TL} \models_{X \rightarrow S} (\psi_2)^{TL}$, that is, by induction hypothesis, if and only if $M \models_X \psi_1 \wedge \psi_2$. 
\item $M^{TL} \models_{X \rightarrow S} \exists v;(\psi)^{TL}$ if and only if there exists a $Y$ such that $Y \supseteq X[F/v]$ for some $F$ and $M^{TL} \models_{Y \rightarrow S} \psi$. By induction hypothesis and downwards closure, this can be the case if and only if $M \models_{X[F/v]} \psi$ for some $F$, that is, if and only if $M \models_X \exists v \psi$; 
\item $M^{TL} \models_{X \rightarrow S} \forall v; (\psi)^{TL}$ if and only if $M^{TL} \models_{Y \rightarrow S} (\psi)^{TL}$ for some $Y \supseteq X[M/v]$, that is, if and only if $M \models_{X[M/v]} \psi$, that is, if and only if $M \models_X \forall v \psi$. 
\end{enumerate}
\end{proof}
One interesting aspect of this representation result is that Dependence Logic \emph{formulas} correspond to Transition Logic \emph{transitions}, not to Transition Logic \emph{formulas}. This can be thought of as one first hint of the fact that Dependence Logic can be thought of as a logic of transitions: and in the later sections, we will explore this idea more in depth.

Representing Transition Models, game terms and formulas in Dependence Logic is somewhat more complex:
\begin{defin}[$T^{DL}$]
Let $T = (S, (\theta_t : t \in \Theta), V)$ be a transition model. Furthermore, for any $t \in \Theta$, let $\theta_t = \{(X_i, Y_i) : i \in I_t\}$, and, for any $p \in \Phi$, let $V(p) = \{X_j : j \in J_p\}$. Then $T^{DL}$ is the first-order model with domain\footnote{Here we write $A \uplus B$ for the \emph{disjoint union} of the sets $A$ and $B$.} $S \uplus \biguplus \{I_t : t \in \Theta\} \uplus \biguplus \{J_p : p \in \Phi\}$ whose signature contains 
\begin{itemize}
\item For every $t \in \Theta$, a ternary relation $R_t$ whose interpretation is $\{(i, x, y) : i \in I_t, x \in X_i, y \in Y_i\}$;
\item For every $p \in \Phi$, a binary relation $V_p$ whose interpretation is $\{(j, x) : j \in J_p, x \in X_j\}$. 
\end{itemize}
\end{defin}
\begin{defin}[$\phi^{DL}_x$ and $\tau^{DL}_x$]
For any transition formula $\phi$ and variable $x$, the Dependence Logic formula $\phi^{DL}_x$ is defined as
\begin{enumerate}
\item $\top^{DL}_x$ is $\top$;
\item For all $p \in \Phi$, $p^{DL}_x$ is $\exists j (=\!\!(j) \wedge V_p(j, x))$;
\item $(\psi_1 \vee \psi_2)^{DL}_x$ is $(\psi_1)^{DL}_x \sqcup (\psi_2)^{DL}$, where $\sqcup$ is the classical disjunction introduced in Definition \ref{defin:classic_or};
\item $(\psi_1 \wedge \psi_2)^{DL}_x$ is $(\psi_1)^{DL}_x \wedge (\psi_2)^{DL}_x$; 
\item $(\langle \tau\rangle \psi)^{DL}_x$ is $\exists P ((\tau)^{DL}_x(P) \wedge \forall y (\lnot Py \vee (\psi)^{DL}_y))$,
\end{enumerate}
where for any transition term $\tau$, variable $x$ and unary relation symbol $P$, $\tau^{DL}_x(P)$ is defined as
\begin{enumerate}
\setcounter{enumi}{5}
\item For all $t \in \Theta$, $t^{DL}_x(P)$ is $\exists i (=\!\!(i) \wedge \exists y (R_t(i, x, y)) \wedge \forall y(\lnot R_t(i, x, y) \vee Py))$;
\item For all formulas $\phi$, $(\phi?)^{DL}_x(P)$ is $\phi^{DL}_x \wedge Px$;
\item $(\tau_1 \otimes \tau_2)^{DL}_x(P) = (\tau_1)^{DL}_x(P) \vee (\tau_2)^{DL}_x(P)$; 
\item $(\tau_1 \cap \tau_2)^{DL}_x(P) = (\tau_1)^{DL}_x(P) \wedge (\tau_2)^{DL}_x(P)$; 
\item $(\tau_1;\tau_2)^{DL}_x(P) = \exists Q((\tau_1)^{DL}_x(Q) \wedge \forall y (\lnot Qy \vee (\tau_2)^{DL}_y(P)))$ for a new and unused variable $y$. 
\end{enumerate}
\end{defin}
\begin{theo}
\label{theo:TL-Rep2}
For all transition models $T = (S, (\theta_t : t \in \Theta), V)$, transition terms $\tau$, transition formulas $\phi$, variables $x$, sets $P \subseteq S$ and teams $X$ over $T^{DL}$ with $X(x) \subseteq S$,\footnote{That is, such that $X(x)$ is a set of states of the transition model.} 
\[
	T^{DL} \models_X \phi^{DL}_x \Leftrightarrow T \models_{X(x)} \phi
\]
and 
\[
	T^{DL} \models_X \tau^{DL}_x(P) \Leftrightarrow T \models_{X(x) \rightarrow P} \tau.
\]
\end{theo}
\begin{proof}
The proof is by structural induction on terms and formulas. 

Let us first consider the cases corresponding to formulas: 
\begin{enumerate}
\item For all teams $X$, $T^{DL} \models_X \top$ and $T \models_{X(x)} \top$, as required;
\item Suppose that $T^{DL} \models_X \exists j(=\!\!(j) \wedge V_p(j, x))$. Then there exists a $m \in \domain(T^{DL})$ such that $T^{DL} \models_{X[m/j]} V_p(j, x)$. Hence, we have that $X(x) \subseteq X_m \in V(p)$; and, by downwards closure, this implies that $X(x) \in V(p)$, and hence that $T \models_{X(x)} p$ as required.

Conversely, suppose that $T \models_{X(x)} p$. Then $X(x) \in V(p)$, and hence $X(x) = X_m$ for some $m \in J_p$. Then we have by definition that $T^{DL} \models_{X[m/j]} V_p(j, x)$, and finally that $T^{DL} \models_{X} T_x(p)$. 
\item By Proposition \ref{propo:classic_or}, $T^{DL} \models_X (\psi_1 \vee \psi_2)^{DL}_x$ if and only if $T^{DL} \models_X (\psi_1)^{DL}_x$ or $T^{DL} \models_X (\psi_2)^{DL}_x$. By induction hypothesis, this is the case if and only if $T \models_{X(x)} \psi_1$ or $T \models_{X(x)} \psi_2$, that is, if and only if $T \models_{X(x)} \psi_1 \vee \psi_2$. 
\item $T^{DL} \models_X (\psi_1 \wedge \psi_2)^{DL}_x$ if and only if $T^{DL} \models_X (\psi_1)^{DL}_x$ and $T^{DL} \models_x (\psi_2)^{DL}_x$, that is, by induction hypothesis, if and only if $T \models_X \psi_1 \wedge \psi_2$. 
\item $T^{DL} \models_X (\langle \tau \rangle \psi)^{DL}_x$ if and only if there exists a $P$ such that $T^{DL} \models_X (\tau)^{DL}_x(P)$ and $T^{DL} \models_{X[T^{DL}/y]} \lnot Py \vee (\psi)^{DL}_y$. By induction hypothesis, the first condition holds if and only if $T \models_{X(x) \rightarrow P} \tau$. As for the second one, it holds if and only if $X[T^{DL}/y] = Y_1 \cup Y_2$ for two $Y_1$, $Y_2$ such that $T^{DL} \models_{Y_1} \lnot Py$ and $T^{DL} \models_{Y_2} \tau_y(\psi)$. But then we must have that $T \models_{Y_2(y)} \psi$ and that $P \subseteq Y_2(y)$; therefore, by downwards closure, $T \models_P \psi$ and finally $T \models_{X(x)} \langle \tau\rangle \psi$. 

Conversely, suppose that there exists a $P$ such that $T \models_{X(x) \rightarrow P} \tau$ and $T \models_{P} \psi$; then by induction hypothesis we have that $T^{DL} \models_{X} (\tau)^{DL}_x(P)$ and that $T^{DL} \models_{X[T^{DL}/y]} \lnot Py \vee (\psi)^{DL}_x$, and hence $T^{DL} \models_X (\langle \tau\rangle \psi)^{DL}_x$. 
\end{enumerate}
Now let us consider the cases corresponding to transition terms:
\begin{enumerate}
\setcounter{enumi}{5}
\item Suppose that $T^{DL} \models_X \exists i (=\!\!(i) \wedge \exists y (R_t(i, x, y)) \wedge \forall y(\lnot R_t(i, x, y) \vee Py))$. If $X = \emptyset$ then $X(x) = \emptyset$, and hence by \textbf{non-creation} we have that $(X(x), P) = (\emptyset, P) \in \theta_t$, as required. 

Let us assume instead that $X \not = \emptyset$. Then, by hypothesis, there exists a $m \in \domain(T^{DL})$ such that 
\begin{itemize}
\item There exists a $F$ such that $T^{DL} \models_{X[m/i][F/y]} R_t(i, x, y)$;
\item $T^{DL} \models_{X[m/i][T^{DL}/y]} \lnot R_t(i, x, y) \vee Py$. 
\end{itemize}
From the first condition it follows that for every $p \in X(x)$ there exists a $q$ such that $R_t(m, p, q)$: therefore, by the definition of $R_t$, every such $p$ must be in $X_m$. 

From the second condition it follows that whenever $R_t(m, p, q)$ and $p \in X(x) \subseteq X_m$, $q \in P$; and, since $X(x) \not = \emptyset$, this implies that $Y_m \subseteq P$ by the definition of $R_t$.

Hence, by \textbf{monotonicity} and \textbf{downwards closure}, we have that $(X(x), P) \in \theta_t$ and that $T \models_{X(x) \rightarrow P} t$, as required. 

Conversely, suppose that $(X(x), P) = (X_m, Y_m) \in \theta_t$ for some $m \in I_t$. If $X(x) = \emptyset$ then $X = \emptyset$, and hence by Proposition \ref{propo:emptyteam} we have that $T^{DL} \models_X t^{DL}_x(P)$, as required. Otherwise, by \textbf{non-triviality}, $P = Y_m \not = \emptyset$. Let now $p \in P$ be any of its elements and let $F(s) = p$ for all $p \in X[m/i]$: then $M \models_{X[m/i][F/y]} R_t(i, x, y)$, as any assignment of this team sends $x$ to some element of $X_m$ and $y$ to $p \in Y_m$. Furthermore, let $s \in X(x) = X_m$, and let $q$ be such that $R_t(m, s(x), q)$: then $q \in Y_m = P$, and hence $M \models_{X[m/i][T^{DL}/y]} \lnot R_t(i, x, y) \vee Py$. So, in conclusion, $M \models_X t^{DL}_x(P)$, as required.
\item $T^{DL} \models_X \phi^{DL}_x \wedge Px$ if and only if $T \models_{X(x)} \phi$ and $X(x) \subseteq P$, that is, if and only if $T \models_{X(x) \rightarrow P} \phi?$.
\item $T^{DL} \models_X (\tau_1)^{DL}_x(P) \vee (\tau_2)^{DL}_x(P)$ if and only if $X = X_1 \cup X_2$ for two $X_1, X_2$ such that 
\begin{itemize}
\item $X = X_1 \cup X_2$, and therefore $X(x) = X_1(x) \cup X_2(x)$;
\item $T^{DL} \models_{X_1} (\tau_1)^{DL}_x(P)$, that is, by induction hypothesis, $T \models_{X_1(x) \rightarrow P} \tau_1$; 
\item $T^{DL} \models_{X_2} (\tau_2)^{DL}_x(P)$, that is, by induction hypothesis, $T \models_{X_2(x) \rightarrow P} \tau_2$; 
\end{itemize}
Hence, if $T^{DL} \models_X (\tau_1 \otimes \tau_2)^{DL}_x(P)$ then $T \models_{X(x) \rightarrow P} \tau_1 \otimes \tau_2$. 

Conversely, if $X(x) = A \cup B$ for two $A$, $B$ such that $T \models_{A \rightarrow P} \tau_1$ and $T \models_{B \rightarrow P} \tau_2$, let 
\begin{eqnarray*}
X_1 &=& \{s \in X : s(x) \in A\}\\
X_2 &=& \{s \in X : s(x) \in B\}.
\end{eqnarray*}
Clearly $X = X_1 \cup X_2$, and furthermore by induction hypothesis $T^{DL} \models_{X_1} (\tau_1)^{DL}_x(P)$ and $T^{DL} \models_{X_2} (\tau_2)^{DL}_x(P)$. Hence, $T^{DL} \models_X (\tau_1 \otimes \tau_2)^{DL}_x(P)$, as required.
\item $T^{DL} \models_{X} (\tau_1 \cap \tau_2)^{DL}_x(P) $ if and only if $T^{DL} \models_X (\tau_1)^{DL}_x(P)$ and $T^{DL} \models_X (\tau_2)^{DL}_x(P)$, that is, by induction hypothesis, if and only if $T \models_{X(x) \rightarrow P} \tau_1 \cap \tau_2$.
\item $T^{DL} \models_X \exists Q((\tau_1)^{DL}_x(Q) \wedge \forall y (\lnot Qy \vee (\tau_2)^{DL}_y(P)))$ if and only if there exists a $Q$ such that $T \models_{X(x) \rightarrow Q} \tau_1$ and there exists a $Q' \supseteq Q$ such that $T \models_{Q' \rightarrow P} \tau_2$. By downwards closure, if this is the case then $T \models_{Q \rightarrow P} \tau_2$ too, and hence $T \models_{X(x) \rightarrow P} \tau_1;\tau_2$, as required. 

Conversely, suppose that there exists a $Q$ such that $T \models_{X(x) \rightarrow Q} \tau_1$ and $T \models_{Q \rightarrow P} \tau_2$. Then, by induction hypothesis $T^{DL} \models_{X} (\tau_1)^{DL}_x(Q)$; and furthermore, $X[T^{DL}/y]$ can be split into
\[
	Z_1 = \{s \in X[T^{DL}/y] : s(y) \not \in Q\}
\]
and 
\[
	Z_2 = \{s \in X[T^{DL}/y] : s(y) \in Q\}
\]
It is trivial to see that $T^{DL} \models_{Z_1} \lnot Qy$; and furthermore, since $Z_2(y) = Q$ and $T \models_{Q \rightarrow P} \tau_2$, by induction hypothesis we have that $T^{DL} \models_{Z_2} (\tau_2)^{DL}_y$. Thus $T^{DL} \models_{X[T^{DL}/y]} \forall y (\lnot Qy \vee (\tau_2)^{DL}_y(P))$ and finally $T^{DL} \models_X (\tau_1;\tau_2)^{DL}_x(P)$, and this concludes the proof.
\end{enumerate}
\end{proof}
Hence, the relationship between Transition Logic and Dependence Logic is analogous to the one between Dynamic Game Logic and First-Order Logic. In the next sections, we will develop variants of Dependence Logic which are syntactically closer to Transition Logic, while still being first-order: as we will see, the resulting frameworks are expressively equivalent to Dependence Logic on the level of satisfiability, but can be used to represent finer-grained phenomena of \emph{transitions} between sets of assignments. 
\section{Dynamic Variants of Dependence Logic}
\subsection{Dependence Logic and Transitions between Teams}
Now that we have established a connection between Dependence Logic and a variant of Dynamic Game Logic, it is time to explore what this might imply for the further development of logics of imperfect information. If, as Theorems \ref{theo:TL-Rep1} and \ref{theo:TL-Rep2} suggest, Dependence Logic can be thought of as a logic of imperfect-information decision problems, perhaps it could be possible to develop variants of Dependence Logic in which expressions can be interpreted directly as transition systems? 

In what follows, we will do exactly that, first with \emph{Transition Dependence Logic} -- a variant of Dependence Logic, expressively equivalent to it, which is also a quantified version of Transition Logic -- and then with \emph{Dynamic Dependence Logic}, in which \emph{all} expressions are interpreted as transitions! 

But why would we interested in such variants of Dependence Logic? One possible answer, which we will discuss in this subsection, is that transitions between teams are \emph{already} a central object of study in the field of Dependence Logic, albeit in a non-explicit manner: after all, the semantics of Dependence Logic interprets quantifiers in terms of transformations of teams, and disjunctions in terms of decompositions of teams into subteams. This intuition is central to the study of issues of interdefinability in Dependence Logic and its variants, like for example the ones discussed in \cite{galliani12}. As a simple example, let us recall Definition \ref{defin:classic_or}:
\[
	\psi_1 \sqcup \psi_2 := \exists u_1 \exists u_2 (=\!\!(u_1) \wedge =\!\!(u_2) \wedge ((u_1 = u_2 \wedge \psi_1) \vee (u_1 \not = u_2 \wedge \psi_2))),
\]
where $u_1$ and $u_2$ are new variables.

As we said in Proposition \ref{propo:classic_or}, $M \models_X \psi_1 \sqcup \psi_2$ if and only if $M \models_X \psi_1$ or $M \models_X \psi_2$. We will now sketch the proof of this result, and -- as we will see -- this proof will hinge on the fact that the above expression can be read as a specification of the following algorithm:
\begin{enumerate}
\item Choose an element $a \in \dom(M)$ and extend the team $X$ by assigning $a$ as the value of $u_1$ for all assignments;
\item Choose an element $b \in \dom(M)$ and further extend the team by assigning $b$ as the value of $u_2$ for all assignments;
\item Split the resulting team into two subteams $Y_1$ and $Y_2$ such that 
\begin{enumerate}
\item $\psi_1$ holds in $Y_1$, and the values of $u_1$ and $u_2$ coincide for all assignments in it;
\item $\psi_2$ holds in $Y_2$, and the values of $u_1$ and $u_2$ differ for all assignments in it.
\end{enumerate}
\end{enumerate}
Since the values of $u_1$ and $u_2$ are chosen to always be respectively $a$ and $b$, one of $Y_1$ and $Y_2$ is empty and the other is of the form $X[ab/u_1u_2]$, and since $u_1$ and $u_2$ do not occur in $\psi_1$ or $\psi_2$ the above algorithm can succeed (for some choice of $a$ and $b$) only if $M \models_X \psi_1$ or $M \models_X \psi_2$.

As another, slightly more complicated example, let us consider the following problem. Given four variables $x_1$, $x_2$, $y_1$ and $y_2$, let $x_1 x_2 ~|~ y_1 y_2$ be an \emph{exclusion atom} holding in a team $X$ if and only if for all $s, s' \in X$, $s(x_1 x_2) \not = s'(y_1 y_2)$ -- that is, if and only if the sets of the values taken by $x_1 x_2$ and by $y_1 y_2$ in $X$ are disjoint. 

By Theorem \ref{SigmaToDL}, we can tell at once that there exists some Dependence Logic formula $\phi(x_1, x_2, y_1, y_2)$ such that for all suitable $M$ and $X$, $M \models_X \phi(x_1, x_2, y_1, y_2)$ if and only if $M \models_X x_1 x_2 ~|~ y_1 y_2$; but what about the converse? For example, can we find an expression $\psi(x, y)$, in the language of First Order Logic augmented with these exclusion atoms (but with no dependence atoms), such that for all suitable $M$ and $X$ $M \models_{X} \psi(x, y)$ if and only if $M \models_X =\!\!(x,y)$?

As discussed in \cite{galliani12} in a more general setting, the answer is positive, and one such $\psi(x, y)$ is $\forall z (z = y \vee (z \not = y \wedge x z ~|~ x y))$, where $z$ is some variable other than $x$ and $y$.\footnote{A moment's thought shows that, by downwards closure, the condition $z \not = y$ in the second disjunct can be removed, but for simplicity we will keep it.} Why is this the case? 

Well, let us consider any team $X$ with domain containing $x$ and $y$, and let us evaluate $\psi(x,y)$ over it. As shown graphically in Figure \ref{fig:F1}, the transitions between teams occurring during the evaluation of the formula correspond to the following algorithm: 
\begin{enumerate}
\item First, assign all possible values to the variable $z$ for all assignments in $x$, thus obtaining $X[M/z] = \{s[m/z] : s \in X, m \in \dom(M)\}$; 
\item Then, remove from $X[M/z]$ all assignments $s$ for which $s(z) = s(x)$, keeping only the ones for which $s(z) \not = s(y)$; 
\item Then, verify that for any possible fixed value of $x$, the possible values of $y$ and $z$ are disjoint. 
\end{enumerate}
This algorithm succeeds only if $y$ is a function of $x$. Indeed, suppose that instead there are two assignments $s, s' \in X$ such that $s(x) = s'(x) = a$, $s(y) = b$ and $s'(y) = c$ for three $a, b, c \in \dom(M)$ with $b \not = c$. Now we have that $\{s[b/z], s[c/z], s'[b/z], s'[c/z]\} \subseteq X[M/z]$: and since $b \not = c$, we have that the assignments $s[c/z]$ and $s'[b/z]$ are not removed from the team in the second step of the proof. But then $s[c/z](xz) = a c = s'[b/z](xy)$, and therefore it is not true that $xy ~|~ xz$. And, conversely, if in the team $X$ the value of $y$ is a function of the value of $x$ then by splitting $X[M/z]$ into the two subteams $Y = \{s[m/z] : s \in X, s(y) = s(z)\}$ and $Z = \{s[m/z] : s(y) \not = s(z)\}$ we have that $M \models_Y y = z$, $M \models_Z y \not = z$ and $M \models_Z xz ~|~ xy$ (since for all $s, s' \in Z$, $s(x) = s'(x) \Rightarrow s(y) = s'(y) \Rightarrow s(z) \not = s(y) = s'(z)$). 

\begin{figure}
\center{
	\epsfig{file=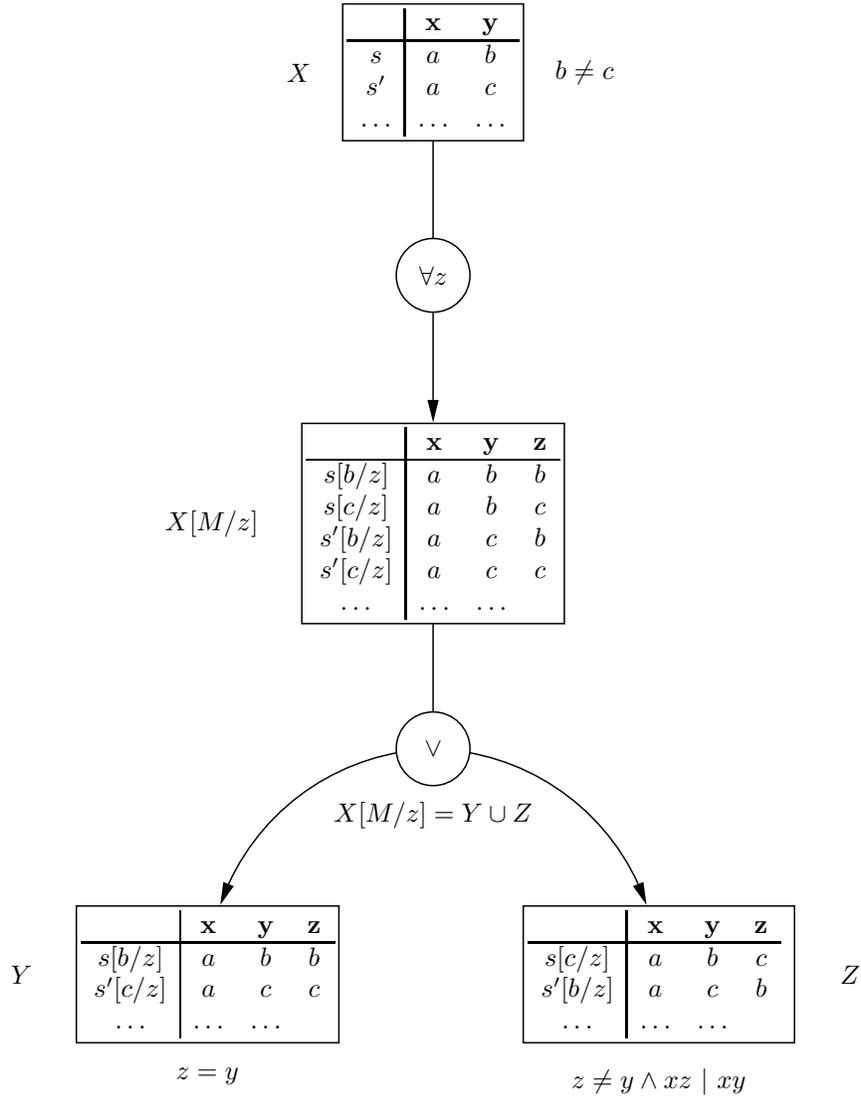}
}
\caption{Checking $=\!\!(x,y)$ by evaluating $\forall z (z = y \vee (z \not = y \wedge xz ~|~ xy))$. If $b \not = c$, then $Z$ does not satisfy the $xz ~|~ xy$.}
\label{fig:F1}
\end{figure}

On the other hand, one Dependence Logic expression corresponding to $x_1 x_2 ~|~ y_1 y_2$ is 
\[
\begin{array}{l}
 \forall w_1 w_2 \exists u_1 u_2 (=\!\!(w_1,w_2,u_1) \wedge =\!\!(w_1, w_2, u_2) \wedge\\
 ~ ~ ~ ((u_1 = u_2 \wedge (w_1 \not = x_1 \vee w_2 \not = z_2)) \vee\\
~ ~ ~ ~ (u_1 \not = u_2 \wedge (w_1 \not = y_1 \vee w_2 \not = y_2))))
\end{array}
\]
where $w_1$, $w_2$, $u_1$ and $u_2$ are new variable. 

We encourage the interested reader to verify that this is the case by examining the transitions between teams corresponding to the formula: in brief, the intuition is that first we extend our team by picking all possible pairs of values for $w_1$ and $w_2$, then for any such pair we flag -- through our choice of $u_1$ and $u_2$ -- whether $w_1 w_2$ is different from $x_1 x_2$ or from $y_1 y_2$. This implies that no such pair is equal to both $x_1 x_2$ and $y_1 y_2$, or, in other words, that $x_1 x_2$ and $y_1 y_2$ have no value in common. 

More and more complex examples of definability results of this kind can be found in \cite{galliani12}; but what we want to emphasize here is that all these examples, like the one we discussed in depth here, have a natural interpretation in terms of algorithms which transform teams and apply simple tests to them, as the above one. Hence, we hope that the development of variants of Dependence Logic in which these transitions are made explicit might prove itself useful for the further study of this interesting class of problems.
\subsection{Transition Dependence Logic}
\label{subsect:TDL}
%
%
As stated, we will now define a variant of Dependence Logic which can also be seen as a quantified variant of Transition Logic. We will then prove that the resulting Transition Dependence Logic is expressively equivalent to Dependence Logic, in the sense that any Dependence Logic formula is equivalent to some Transition Dependence Logic formula and vice versa.
\begin{defin}[Transition Dependence Logic - Syntax]
Let $\Sigma$ be a first-order signature. Then the sets of all \emph{transition terms} and of all \emph{formulas} of Dependence Transition Logic are given by the rules
\begin{eqnarray*}
\tau &::=& \exists v ~|~ \forall v ~|~ \phi? ~|~ \tau \otimes \tau ~|~ \tau \cap \tau ~|~ \tau;\tau\\
\phi &::=& R\tuple t ~|~ \lnot R \tuple t ~|~ =\!\!(t_1, \ldots, t_n) ~|~ \phi \vee \phi ~|~ \phi \wedge \phi ~|~ \langle \tau \rangle \phi.
\end{eqnarray*}
where $v$ ranges over all variables in $\var$, $R$ ranges over all relation symbols of the signature, $\tuple t$ ranges over all tuples of terms of the required arities, $n$ ranges over $\mathbb N$ and $t_1 \ldots t_n$ range over the terms of our signature. 
\end{defin}
\begin{defin}[Transition Dependence Logic - Semantics] 
Let $M$ be a first-order model, let $\tau$ be a first-order transition term of the same signature, and let $X$ and $Y$ be teams over $M$. Then we say that the transition $X\rightarrow Y$ is \emph{allowed} by $\tau$ in $M$, and we write $M \models_{X \rightarrow Y} \tau$, if and only if 
\begin{description}
\item[TDL-$\exists$: ] $\tau$ is of the form $\exists v$ for some $v \in \var$ and there exists a $F$ such that $X[F/v]\subseteq Y$;
\item[TDL-$\forall$: ] $\tau$ is of the form $\forall v$ for some $v \in \var$ and $X[M/v] \subseteq Y$;
\item[TDL-test: ] $\tau$ is of the form $\phi?$, $M \models_X \phi$ in the sense given later in this definition, and $X \subseteq Y$;
\item[TDL-$\otimes$: ] $\tau$ is of the form $\tau_1 \otimes \tau_2$ and $X = X_1 \cup X_2$ for some $X_1$ and $X_2$ such that $M \models_{X_1 \rightarrow Y} \tau_1$ and $M \models_{X_2 \rightarrow Y} \tau_2$; 
\item[TDL-$\cap$: ] $\tau$ is of the form $\tau_1 \cap \tau_2$, $M \models_{X \rightarrow Y} \tau_1$ and $M \models_{X\rightarrow Y} \tau_2$; 
\item[TDL-concat: ] $\tau$ is of the form $\tau_1;\tau_2$ and there exists a team $Z$ such that $M \models_{X \rightarrow Z} \tau_1$ and $M \models_{Z \rightarrow Y} \tau_2$.
\end{description}
Similarly, if $\phi$ is a formula and $X$ is a team with domain $\var$. Then we say that $X$ \emph{satisfies} $\phi$ in $M$, and we write $M \models_X \phi$, if and only if 
\begin{description}
\item[TDL-lit: ] $\phi$ is a first-order literal and $M \models_s \phi$ in the usual first-order sense for all $s \in X$; 
\item[TDL-dep: ] $\phi$ is a dependence atom $=\!\!(t_1, \ldots, t_n)$ and any two $s, s' \in X$ which assign the same values to $t_1 \ldots t_{n-1}$ also assign the same value to $t_n$; 
\item[TDL-$\vee$: ] $\phi$ is of the form $\phi_1 \vee \phi_2$ and $M \models_X \phi_1$ or $M \models_X \phi_2$; 
\item[TDL-$\wedge$: ]$\phi$ is of the form $\phi_1 \wedge \phi_2$, $M \models_X \phi_1$ and $M \models_X \phi_2$; 
\item[TDL-$\diamond$: ] $\phi$ is of the form $\langle \tau \rangle \psi$ and there exists a $Y$ such that $M \models_{X \rightarrow Y} \tau$ and $M \models_Y \psi$. 
\end{description}
\end{defin}
As the next theorem shows, in this semantics formulas and transitions are interpreted in terms of trumps and transition systems:
\begin{theo}
For all Transition Dependence Logic formulas $\phi$, all models $M$ and all teams $X$ and $Y$, we have that 
\begin{description}
\item[Downwards Closure:] If $M \models_X \phi$ and $Y \subseteq X$ then $M \models_Y \phi$;
\item[Empty Team Property:] $M \models_\emptyset \phi$. 
\end{description}
Furthermore, for all Transition Dependence Logic transition terms $\tau$, all models $M$ and all teams $X$, $Y$ and $Z$,
\begin{description}
\item[Downwards Closure: ] If $M \models_{X \rightarrow Y} \tau$ and $Z \subseteq X$ then $M \models_{Z \rightarrow Y} \tau$;
\item[Monotonicity: ] If $M \models_{X \rightarrow Y} \tau$ and $Y \subseteq Z$ then $M \models_{X \rightarrow Z} \tau$;
\item[Non-creation: ] For all $Y$, $M \models_{\emptyset \rightarrow Y} \tau$;
\item[Non-triviality: ] If $X \not = \emptyset$ then $M \not \models_{X \rightarrow \emptyset} \tau$.
\end{description}
\end{theo}
\begin{proof}
The proof is by structural induction over $\phi$ and $\tau$, and presents no difficulties whatsoever.
\end{proof}

Also, it is not difficult to see, on the basis of the results of the previous section, that this new variant of Dependence Logic is equivalent to the usual one:
\begin{theo}
For every Dependence Logic formula $\phi$ there exists a Transition Dependence Logic transition term $\tau_\phi$ such that 
\[
	M \models_X \phi \Leftrightarrow \exists Y \mbox{ s.t. } M \models_{X \rightarrow Y} \tau_\phi \Leftrightarrow M \models_X \langle \tau_\phi\rangle \top
\]
for all first-order models $M$ and teams $X$. 
\end{theo}
\begin{proof}
$\tau_\phi$ is defined by structural induction on $\phi$, as follows: 
\begin{enumerate}
\item If $\phi$ is a first-order literal or a dependence atom then $\tau_\phi = \phi?$; 
\item If $\phi$ is $\phi_1 \vee \phi_2$ then $\tau_\phi = \tau_{\phi_1} \otimes \tau_{\phi_2}$; 
\item If $\phi$ is $\phi_1 \wedge \phi_2$ then $\tau_\phi = \tau_{\phi_1} \cap \tau_{\phi_2}$; 
\item If $\phi$ is $\exists v \psi$ then $\tau_\phi = \exists v; \tau_{\psi}$; 
\item If $\phi$ is $\forall v \psi$ then $\tau_\phi = \forall v; \tau_{\psi}$.
\end{enumerate}
It is then trivial to verify, again by induction on $\phi$, that $M \models_X \phi$ if and only if $M \models_{X} \langle \tau_\phi\rangle \top$, as required.
\end{proof}	
This representation result associates Dependence Logic \emph{formulas} to Transition Dependence Logic \emph{transition terms}. This fact highlights the dynamical nature of Dependence Logic operators, which we discussed in the previous subsection: in this framework, quantifiers describe \emph{transformations} of teams, the Dependence Logic connectives are operations over games, and the literals are interpreted as tests. In fact, one might wonder what is the purpose of Transition Dependence Logic formulas: could we do away with them altogether, and develop a variant of Transition Dependence Logic in which \emph{all} formulas are transitions? 

Later, we will explore this idea further; but first, let us verify that Transition Dependence Logic is no more expressive than Dependence Logic.
%
%
%
\begin{theo}
For every Transition Dependence Logic formula $\phi$ there exists a Dependence Logic formula $T(\phi)$ such that 
\[
	M \models_X \phi \Leftrightarrow M \models_X T(\phi)
\]
for all first-order models $M$ and teams $X$. Furthermore, for every Transition Dependence Logic transition term $\tau$ and Dependence Logic formula $\theta$ there is a Dependence Logic formula $U(\tau, \psi)$ such that 
\[
	M \models_X U(\tau, \theta) \Leftrightarrow \exists Y \mbox{ s.t. } M \models_{X \rightarrow Y} \tau \mbox{ and } M \models_Y \theta,
\]
again for all first-order models $M$ and teams $X$.
\end{theo}
\begin{proof}
We prove the two claims together, by structural induction over $\phi$ and $\tau$. 

First, let us consider the cases corresponding to formulas: 
\begin{enumerate}
\item If $\phi$ is a first order literal or a dependence atom, let $T(\phi)$ be $\phi$ itself. As the interpretation of these expressions is the same in Dependence Logic and in Transition Dependence Logic, there is nothing to prove. 
\item If $\phi$ is of the form $\psi_1 \vee \psi_2$, let $T(\phi)$ be $T(\psi_1) \sqcup T(\psi_2)$. 
This expression holds in a team if and only if $T(\psi_1)$ or $T(\psi_2)$ hold, that is, by induction hypothesis, if and only if $\psi_1$ or $\psi_2$ do.
\item If $\phi$ is of the form $\psi_1 \wedge \psi_2$, let $T(\phi)$ be $T(\psi_1) \wedge T(\psi_2)$. Then $T(\phi)$ holds if and only if $\psi_1$ and $\psi_2$ do, that is, if and only if $\phi$ does. 
\item If $\phi$ is of the form $\langle \tau \rangle \psi$, let $\tuple v$ be the tuple of all variables occurring in $T(\psi)$, let $R$ be a new $|\tuple v|$-ary relation, and let $T(\phi)$ be $\exists R (U(\tau, R\tuple v) \wedge \forall \tuple v (\lnot R \tuple v \vee T(\psi)))$. Indeed, suppose that $M \models_X T(\phi)$: then for some relation $R$, there exists a $Y$ such that $M \models_{X \rightarrow Y} \tau$ and $M \models_Y R\tuple v$. Furthermore, $M \models \forall \tuple v(\lnot R \tuple v \vee T(\psi))$, and therefore for the set $Y' = \{s : \dom(s) = \tuple v, M \models_s R \tuple v\}$ we have that $M \models_{Y'} T(\psi)$. But then, by downwards closure and locality, $M \models_Y T(\psi)$, and therefore $M \models_X \langle \tau\rangle \psi$. 

Conversely, suppose that $M \models_X \langle \tau \rangle \psi$: then there exists a $Y$ such that $M \models_{X \rightarrow Y} \tau$ and $M \models_Y \psi$. Now let $R$ be $\{s(\tuple v) : s \in Y\}$: clearly $M \models_X U(\tau, R\tuple v)$, since $M \models_Y R \tuple v$, and furthermore $M \models \forall \tuple v (\lnot R \tuple v \vee T(\psi))$, by locality and by the fact that (by induction hypothesis) $M \models_Y T(\psi)$. 
\end{enumerate}
Now let us consider the cases corresponding to transitions: 
\begin{enumerate}
\setcounter{enumi}{4}
\item If $\tau$ is of the form $\exists v$ for some variable $v$, let $U(\tau, \theta)$ be $\exists v \theta$. Indeed, suppose that $M \models_X \exists v \theta$: then $M \models_{X[F/v]} \theta$ for some $F$, and by choosing $Y = X[F/v]$ we have that $M \models_{X \rightarrow Y} \exists v$ and $M \models_Y \theta$, as required. Conversely, suppose that for some $Y$, $M \models_{X \rightarrow Y} \exists v$ and $M \models_Y \theta$: then for some $F$, $X[F/v] \subseteq Y$, and by downwards closure we have that $M \models_{X[F/v]} \theta$.
\item If $\tau$ is of the form $\forall v$ for some variable $v$, let $U(\tau, \theta)$ be $\forall v \theta$. Indeed, suppose that $M \models_X \forall v \theta$: then $M \models_{X[M/v]} \theta$, and if we choose $Y = X[M/v]$ we have at once that $M \models_{X \rightarrow Y} \forall v$ and $M \models_Y \theta$. Conversely, if for some $Y$ $M \models_{X \rightarrow Y} \forall v$ and $M \models_Y \theta$ then $X[M/v] \subseteq Y$ and, by downwards closure, $M \models_{X[M/v]} \theta$. 
\item If $\tau$ is of the form $\phi?$, let $U(\tau, \theta)$ be $T(\phi) \wedge \theta$. Indeed, suppose that $M \models_X T(\phi) \wedge \theta$: then by induction hypothesis $M \models_X \phi$, and, for $Y = X$, we have that $M \models_{X \rightarrow Y} \phi?$. Furthermore, $M \models_Y \theta$, as required. Conversely, suppose that for some $Y$, $M \models_{X \rightarrow Y} \phi?$ and $M \models_Y \theta$. Then $M \models_X \phi$, and therefore $M \models_X T(\phi)$; and furthermore $X \subseteq Y$, and hence by downwards closure $M \models_X \theta$. Hence, $M \models_{X} T(\phi) \wedge \theta$. 
\item If $\tau$ is of the form $\tau_1 \otimes \tau_2$ and $\tuple v$ is the tuple of all free variables of $\theta$ then let $U(\tau, \theta)$ be $\exists R((U(\tau_1, R\tuple v) \vee U(\tau_2, R\tuple v)) \wedge \forall \tuple v (\lnot R \tuple v \vee \theta))$, where $R$ is a new $|\tuple r|$-ary relation symbol. Indeed, suppose that $M \models_X U(\tau, \theta)$: then there exists a relation $R$ and two subteams $X_1$ and $X_2$ of $X$ such that $X = X_1 \cup X_2$, $M \models_{X_1} U(\tau_1, R\tuple v)$ and $M \models_{X_2} U(\tau_2, R\tuple v)$. Hence, there are two teams $Y_1$ and $Y_2$ such that $M \models_{X_1 \rightarrow Y_1} \tau_1$, $M \models_{X_2 \rightarrow Y_2} \tau_2$, $M \models_{Y_1} R \tuple v$ and $M \models_{Y_2} R \tuple v$. Now, let $Y$ be $Y_1 \cup Y_2$: by monotonicity, we have that $M \models_{X_1 \rightarrow Y} \tau_1$ and $M \models_{X_2 \rightarrow Y} \tau_2$, and furthermore $M \models_Y R \tuple v$ too (that is, for all $s \in Y$, $s(\tuple v)$ is in $R$). Since $M \models \forall \tuple v (\lnot R \tuple v \vee \theta)$, this implies that $M \models_Y \theta$, by locality and downwards closure.

Conversely, suppose that there is a $Y$ such that $M \models_{X \rightarrow Y} \tau_1 \otimes \tau_2$ and $M \models_Y \theta$. Then let $R$ be $\{s(\tuple v): s \in Y\}$. Now $X = X_1 \cup X_2$ for two $X_1$ and $X_2$ such that $M \models_{X_1 \rightarrow Y} \tau_1$ and $M \models_{X_2 \rightarrow Y} \tau_2$, and by induction hypothesis we have that $M \models_{X_1} U(\tau_1; R\tuple v)$ and $M \models_{X_2} U(\tau_2; R \tuple v)$. But then $M \models_X U(\tau_1; R \tuple v) \vee U(\tau_2; R \tuple v)$; and furthermore, by locality we have that $M \models \forall \tuple v (\lnot R \tuple v \vee \theta)$. Hence, $M \models_X U(\tau_1 \otimes \tau_2, \theta)$, as required.

\item If $\tau$ is of the form $\tau_1 \cap \tau_2$ and $\tuple v$ is the tuple of all variables of $\theta$ then let $U(\tau, \theta)$ be $\exists \tuple R(U(\tau_1, R \tuple v) \wedge U(\tau_2, R \tuple v) \wedge \forall \tuple v (\lnot R \tuple v \vee \theta))$. Indeed, suppose that $M \models_X U(\tau, \theta)$: then for some relation $R$, by induction hypothesis, there exist teams $Y_1$ and $Y_2$ such that $M \models_{X \rightarrow Y_1} \tau_1$, $M \models_{X \rightarrow Y_2} \tau_2$, $M \models_{Y_1} R \tuple v$ and $M \models_{Y_2} R \tuple v$. Now let $Y$ be $Y_1 \cup Y_2$: as before, by monotonicity we have that $M \models_{X \rightarrow Y} \tau_1$ and $M \models_{X \rightarrow Y} \tau_2$, and hence $M \models_{X \rightarrow Y} \tau_1 \cap \tau_2$. Finally, since $M \models \forall \tuple v (\lnot R \tuple v \vee \theta)$ we have that $M \models_{Y} \theta$, as required. 

Conversely, suppose that there is a $Y$ such that $M \models_{X \rightarrow Y} \tau_1 \cap \tau_2$ and $M \models_Y \theta$. Since $M \models_{X \rightarrow Y} \tau_1 \cap \tau_2$, $M \models_{X \rightarrow Y} \tau_1$ and $M \models_{X \rightarrow Y} \tau_2$. Now let $R$ be $\{s(\tuple v) : s \in Y\}$. By induction hypothesis, $M \models_X U(\tau_1, R \tuple t)$ and $M \models_X U(\tau_2, R \tuple t)$; and furthermore, since $M \models_Y \theta$ we have that $M \models \forall \tuple v (\lnot R \tuple v \vee \theta)$. 
\item If $\tau$ is of the form $\tau_1;\tau_2$ let $U(\tau, \theta)$ be $U(\tau_1, U(\tau_2, \theta))$. Indeed, $M \models_X U(\tau_1, U(\tau_2, \theta))$ if and only if there is a $Y$ such that $M \models_{X \rightarrow Y}$ and $M \models_Y U(\tau_2, \theta)$, that is, if and only if there are a $Y$ and a $Z$ such that $M \models_{X \rightarrow Y} \tau_1$, $M \models_{Y \rightarrow Z} \tau_2$ and $M \models_Z \theta$.
\end{enumerate}
\end{proof}

%
However, in a sense, Transition Dependence Logic allows one to consider subtler distinctions than Dependence Logic does. The formula $\forall x \exists y (=\!\!(y, f(x)) \wedge Pxy)$, for example, could be translated as any of 
\begin{itemize}
\item $\langle \forall x;\exists y\rangle (=\!\!(y, f(x)) \wedge Pxy)$;
\item $\langle \forall x; \exists y\rangle \langle =\!\!(y, f(x))?\rangle Pxy$;
\item $\langle \forall x; \exists y\rangle \langle Pxy?\rangle =\!\!(y, f(x))$;
\item $\langle \forall x; \exists y\rangle \langle (Pxy?) \cap (=\!\!(y, f(x))?) \rangle \top$.
\end{itemize}
The intended interpretations of these formulas are rather different, even though they happen to be satisfied by the same teams: and for this reason, Transition Dependence Logic may be thought of as a proper refinement of Dependence Logic even though it has exactly the same expressive power.
\subsection{Dynamic Predicate Logic}
\label{subsect:DPL}
\emph{Dynamic Semantics} is the name given to a family of semantical frameworks which subscribe to the following principle (\cite{groenendijk91}): 
\begin{quote}
	\emph{The meaning of a sentence does not lie in its truth conditions, but rather in the way it changes (the representation of) the information of the interpreter.}
\end{quote}
In various forms, this intuition can be found prefigured in some of the later work of Ludwig Wittgenstein, as well as in the research of philosophers of language such as Austin, Grice, Searle, Strawson and others (\cite{dekker08}); but its formal development can be traced back to the work of Groenendijk and Stokhof about the proper treatment of pronouns in formal linguistics (\cite{groenendijk91}).
\\

We refer to \cite{dekker08} for a comprehensive analysis of the linguistic issues which caused such a development, as well as for a description of the ways in which this framework was adapted in order to model presuppositions, questions/answers and other phenomena; here we will only present a formulation of \emph{dynamic predicate semantics}, the alternative semantics for first-order logic which was developed in the above mentioned paper by Groenendijk and Stokhof.
\begin{defin}[Dynamic Semantics for First-Order Logic]
Let $\phi$ be a first-order formula, let $M$ be a suitable first-order model and let $s$ and $s'$ be two assignments. Then we say that the transition from $s$ to $s'$ is \emph{allowed} by $\phi$ in $M$, and we write $M \models_{s \rightarrow s'} \phi$, if and only if 
\begin{description}
	\item[DPL-atom: ] $\phi$ is an atomic formula, $s = s'$ and $M \models_s \phi$ in the usual sense; 
	\item[DPL-$\lnot$: ] $\phi$ is of the form $\lnot \psi$, $s=s'$ and for all assignments $h$, $M \not \models_{s \rightarrow h} \psi$; 
	\item[DPL-$\wedge$: ] $\phi$ is of the form $\psi_1 \wedge \psi_2$ and there exists an $h$ such that $M \models_{s \rightarrow h} \psi_1$ and $M \models_{h \rightarrow s'} \psi_2$; 
	\item[DPL-$\vee$: ] $\phi$ is of the form $\psi_1 \vee \psi_2$, $s = s'$ and there exists an $h$ such that $M \models_{s \rightarrow h} \psi_1$ or $M \models_{s \rightarrow h} \psi_2$; 
	\item[DPL-$\rightarrow$: ] $\phi$ is of the form $\psi_1 \rightarrow \psi_2$, $s = s'$ and for all $h$ it holds that 
		\[
			M \models_{s \rightarrow h} \psi_1 \Rightarrow \exists h' \mbox{ s.t. } M \models_{h \rightarrow h'} \psi_2;
		\]
	\item[DPL-$\exists$: ] $\phi$ is of the form $\exists x \psi$ and there exists an element $m \in \domain(M)$ such that $M \models_{s[m/x] \rightarrow s'} \psi$; 
	\item[DPL-$\forall$: ] $\phi$ is of the form $\forall x \psi$, $s = s'$ and for all elements $m \in \domain(M)$ there exists an $h$ such that $M \models_{s[m/x] \rightarrow h} \psi$. 
\end{description}
A formula $\phi$ is \emph{satisfied} by an assignment $s$ if and only if there exists an assignment $s'$ such that $M \models_{s \rightarrow s'} \phi$; in this case, we will write $M \models_s \phi$. 
\end{defin}
We will discuss neither the formal properties of this formalism nor its linguistic applications here. All that is relevant for our purposes is that, according to it, formulas are interpreted as \emph{transitions} from assignments to assignments, and furthermore that the rule for conjunction allows us to bind occurrences of a variable of the second conjunct to quantifiers occurring in the first one.\footnote{For example, consider the formula $(\exists x Px) \wedge Qx$: by the rules given, it is easy to see that $M \models_s  (\exists x Px) \wedge Qx$ if and only if $P^M \cap Q^M \not = \emptyset$, that is, if and only if $M \models_s \exists x (Px \wedge Qx)$, differently from the case of Tarski's semantics.}

The similarity between this semantics and our semantics for transition terms should be evident. Hence, it seems natural to ask whether we can adopt, for a suitable variant of Dependence Logic, the following variant of Groenendijk and Stokhof's motto: \\

\begin{quote}
	\emph{The meaning of a formula does not lie in its satisfaction conditions, but rather in the team transitions it allows.}
\end{quote}

From this point of view, \emph{transition terms} are the fundamental objects of our syntax, and formulas can be removed altogether from the language -- although, of course, the tests corresponding to literals and dependence formulas should still be available. As in Groenendijk and Stokhof's logic, satisfaction becomes then a derived concept: in brief, a team $X$ can be said to satisfy a term $\tau$ if and only if there exists a $Y$ such that $\tau$ allows the transition from $X$ to $Y$, or, in other words, if and only if \emph{some} set of non-losing outcomes can be reached from the set $X$ of initial positions in the game corresponding to $\tau$.

In the next section, we will make use of these intuitions to develop another, terser version of Dependence Logic; and finally, we will discuss some implications of this new version for the further developments and for the possible applications of this interesting logical formalism.
\subsection{Dynamic Dependence Logic}
\label{subsect:DDL}
We will now develop a formula-free variant of Transition Dependence Logic, along the lines of Groenendijk and Stockhof's Dynamic Predicate Logic. 
\begin{defin}[Dynamic Dependence Logic - Syntax]
Let $\Sigma$ be a first-order signature. The set of all formulas of Dynamic Dependence Logic over $\Sigma$ is given by the rules 
\[
	\tau ::= R \tuple t ~|~ \lnot R \tuple t ~|~ =\!\!(t_1, \ldots, t_n) ~|~ \exists v ~|~ \forall v ~|~ \tau \otimes \tau ~|~ \tau \cap \tau ~|~ \tau ; \tau
\]
where, as usual, $R$ ranges over all relation symbols of our signature, $\tuple t$ ranges over all tuples of terms of the required lengths, $n$ ranges over $\mathbb N$, $t_1 \ldots t_n$ range over all terms, and $v$ ranges over $\var$. 
\end{defin}
The semantical rules associated to this language are precisely as one would expect: 
\begin{defin}[Dynamic Dependence Logic - Semantics]
\label{DDL-TTS}
Let $M$ be a first-order model, let $\tau$ be a Dynamic Dependence Logic formula over the signature of $M$, and let $X$ and $Y$ be two teams over $M$ with domain $\var$. Then we say that $\tau$ \emph{allows} the transition $X \rightarrow Y$ in $M$, and we write $M \models_{X \rightarrow Y} \tau$, if and only if 
\begin{description}
\item[DDL-lit: ] $\tau$ is a first-order literal, $M \models_s \tau$ in the usual first-order sense for all $s \in X$, and $X \subseteq Y$; 
\item[DDL-dep: ] $\tau$ is a dependence atom $=\!\!(t_1, \ldots, t_n)$, $X \subseteq Y$, and any two assignments $s, s' \in X$ which coincide over $t_1 \ldots t_{n-1}$ also coincide over $t_n$;
\item[DDL-$\exists$: ] $\tau$ is of the form $\exists v$ for some $v \in \var$, and $X[F/v] \subseteq Y$ for some $F: X \rightarrow \domain(M)$; 
\item[DDL-$\forall$: ] $\tau$ is of the form $\forall v$ for some $v \in \var$, and $X[M/v] \subseteq Y$; 
\item[DDL-$\otimes$: ] $\tau$ is of the form $\tau_1 \otimes \tau_2$ and $X = X_1 \cup X_2$ for two teams $X_1$ and $X_2$ such that $M \models_{X_1 \rightarrow Y} \tau_1$ and $M \models_{X_2 \rightarrow Y} \tau_2$;
\item[DDL-$\cap$: ] $\tau$ is of the form $\tau_1 \cap \tau_2$, $M \models_{X \rightarrow Y} \tau_1$ and $M \models_{X \rightarrow Y} \tau_2$; 
\item[DDL-concat: ] $\tau$ is of the form $\tau_1;\tau_2$, and there exists a $Z$ such that $M \models_{X \rightarrow Z} \tau_1$ and $M \models_{Z \rightarrow Y} \tau_2$. 
\end{description}
A formula $\tau$ is said to be \emph{satisfied} by a team $X$ in a model $M$ if and only if there exists a $Y$ such that $M \models_{X \rightarrow Y} \tau$; and if this is the case, we will write $M \models_X \tau$. 
\end{defin}

It is not difficult to see that Dynamic Dependence Logic is equivalent to Transition Dependence Logic (and, therefore, to Dependence Logic).
\begin{propo}
Let $\phi$ be a Dependence Logic formula. Then there exists a Dynamic Dependence Logic formula $\phi'$ which is equivalent to it, in the sense that
\[
	M \models_X \phi \Leftrightarrow M \models_X \phi' \Leftrightarrow \exists Y \mbox{ s.t. } M \models_{X \rightarrow Y} \phi'
\]
for all suitable teams $X$ and models $M$
\end{propo}
\begin{proof}
We build $\phi'$ by structural induction:
\begin{enumerate}
\item If $\phi$ is a literal or a dependence atom then $\phi' = \phi$; 
\item If $\phi$ is $\psi_1 \vee \psi_2$ then $\phi' = \psi_1' \otimes \psi_2'$; 
\item If $\phi$ is $\psi_1 \wedge \psi_2$ then $\phi' = \psi_1' \cap \psi_2'$; 
\item If $\phi$ is $\exists x \psi$ then $\phi' = \exists x ; \psi'$; 
\item If $\phi$ is $\forall x \psi$ then $\phi' = \forall x ; \psi'$.
\end{enumerate}
\end{proof}
\begin{propo}
Let $\tau$ be a Dynamic Dependence Logic formula. Then there exists a Transition Dependence Logic transition term $\tau'$ such that 
\[
	M \models_{X \rightarrow Y} \tau \Leftrightarrow M \models_{X \rightarrow Y} \tau'
\]
for all suitable $X$, $Y$ and $M$, and such that hence 
\[
	M \models_X \tau \Leftrightarrow M \models_X \langle \tau'\rangle \top.
\]
\end{propo}
\begin{proof}
Build $\tau'$ by structural induction: 
\begin{enumerate}
\item If $\tau$ is a literal or dependence atom then $\tau' = \tau?$; 
\item If $\tau$ is of the form $\exists v$ or $\forall v$ then $\tau' = \tau$;
\item If $\tau$ is of the form $\tau_1 \otimes \tau_2$ then $\tau' = \tau_1' \otimes \tau_2'$; 
\item If $\tau$ is of the form $\tau_1 \cap \tau_2$ then $\tau' = \tau_1' \cap \tau_2'$; 
\item If $\tau$ is of the form $\tau_1;\tau_2$ then $\tau' = \tau_1';\tau_2'$.
\end{enumerate}
\end{proof}
\begin{coro}
Dynamic Dependence Logic is equivalent to Transition Dependence Logic and to Dependence Logic
\end{coro}
\begin{proof}
Follows from the two previous results and from the equivalence between Dependence Logic and Transition Dependence Logic.
\end{proof}
\section{Further Work}
In this work, we established a connection between a variant of Dynamic Game Logic and Dependence Logic, and we used it as the basis for the development of variants of Dependence Logic in which it is possible to talk directly about transitions from teams to teams. This suggests a new perspective on Dependence Logic and Team Semantics, one which allow us to study them as a special kind of \emph{algebras of nondeterministic transitions between relations}. One of the main problems that is now open is whether it is possible to axiomatize these algebras, in the same sense in which, in \cite{mann09}, Allen Mann offers an axiomatization of the algebra of trumps corresponding to IF Logic (or, equivalently, to Dependence Logic). 

Furthermore, we might want to consider different choices of connectives, like for example ones related to the theory of database transactions. The investigation of the relationships between the resulting formalisms is a natural continuation of the currently ongoing work on the study of the relationship between various extensions of Dependence Logic, and promises of being of great utility for the further development of this fascinating line of research. 
\section{Acknowledgements}
The author wishes to thank Johan van Benthem and Jouko V\"a\"an\"anen for a number of useful suggestions and insights. Furthermore, he wishes to thank the reviewers for a number of highly useful suggestions and comments. 
\bibliographystyle{plain}

\end{document}